\documentclass{article}
\usepackage{amscd}
\usepackage{amsmath,amsthm}
\usepackage{amssymb}
\usepackage{graphicx}

\newcommand{\logand}{\;\;{\rm and }\;\;}

\newcommand{\C}{\mathbb{C}}

\newcommand{\R}{\mathbb{R}}

\newcommand{\bigoh}{\mathop{\mathcal{O}}\nolimits}

\newcommand{\spec}{\;{\rm spec }\;}

\newtheorem{theorem}{Theorem}[section]
\newtheorem{lemma}[theorem]{Lemma}
\newtheorem{proposition}[theorem]{Proposition}
\newtheorem{corollary}[theorem]{Corollary}
\theoremstyle{definition}
\newtheorem{definition}{Definition}[section]

\renewcommand{\th}{\ensuremath{^{th}}}
\renewcommand{\Re}{\operatorname{Re}}
\renewcommand{\Im}{\operatorname{Im}}

\title{Signal Velocity in Oscillator Arrays}
\author{ Carlos E. Cantos \thanks{Maseeh Department of Mathematics \& Statistics, Portland State
University, Portland, OR. Email : ccantos@pdx.edu }
\and J. J. P. Veerman \thanks{
Maseeh Department of Mathematics \& Statistics, Portland State
University, Portland, OR, and CCQCN, Dept of Physics, University of Crete, Heraklion, Greece.
Email : \mbox{veerman@pdx.edu}
} 
\and David K. Hammond \thanks{
Applied Mathematics Department, Oregon Institute of Technology,
 Wilsonville, OR. Email : david.hammond@oit.edu}
}

\begin{document}
\maketitle


\begin{abstract} 
  We investigate a system of coupled oscillators on the circle, which
  arises from a simple model for behavior of large numbers of
  autonomous vehicles. The model considers asymmetric, linear,
  decentralized dynamics, where the acceleration of each vehicle
  depends on the relative positions and velocities between itself and
  a set of local neighbors.
  We first derive necessary and sufficient conditions for asymptotic
  stability, then derive expressions for the phase velocity of
  propagation of disturbances in velocity through this system. We show
  that the high frequencies exhibit damping, which implies existence of
  well-defined \emph{signal velocities} $c_+>0$ and $c_-<0$ such that
  low frequency disturbances travel through the flock as
  $f(x-c_+t)$ in the direction of increasing agent numbers and
  $f(x-c_-t)$ in the other.
\end{abstract}

\section{Introduction} 
\label{chap:zero}



This paper is part of a larger program to develop mathematical methods
to quantitatively study performance of models for flocking. The main underlying
motivation for the current work is to inform development of methods for
programming driverless cars to enable coherent motion at high speed, 
even under dense traffic conditions. This is obviously an important problem, 
not only because it can lead to enormous cost savings to have smooth 
and dense traffic on our busier highways, but also because failures may cost lives.

We study models that assume that each car is programmed identically and
that can observe relative velocities and positions of nearby
cars. In this work we take nearby to mean only the car in front and
behind. However the methods we develop will be applicable to larger
interactions (and these will be explored in future work).  We will
assume that the system is linearized.  Various examples and analyses
of nonlinear systems exist. But the emphasis here is on linear systems
where we can allow for many parameters (to take the neighbors into
account) and still perform a meaningful analysis.

There are two main aspects in our analysis. The first is the
asymptotic stability. This can be analyzed via the eigenvalues of the
matrix associated with the first order differential equation. Section
\ref{chap:stability} is devoted to establishing necessary and
sufficient conditions for a class of systems to be asymptotically
stable. Even though this is a fairly straightforward calculation, we
have not found it in this generality in the literature.  

The second, more delicate aspect of the problem is related to the fact
that we may have arbitrarily many cars following each other, hundreds
or even thousands. In this situation, even if all our systems are
known to be asymptotically stable, transients may still grow
exponentially in the number of cars. The spectrum of the linear
operator does not help us to recognize this problem (\cite{Tr}). A
dramatic example of this can be found in \cite{position} where
eigenvalues have real part bounded from above by a negative number and
yet transients grow exponentially in $N$. This kind of exponential
growth underscores the need for different (non-spectral) methods to
analyze these systems. The main result of our paper represents one 
such alternative approach. We establish that  for the parameter
values of interest (e.g. asymptotically stable systems), solutions are well
approximated by travelling wave signals with two distinct signal velocities, one
positive (in direction of increasing agent number) and one negative. 



Ever since the inception (\cite{HM1}, \cite{HM2}) of the subject,
systems with periodic boundary conditions have been popular
(\cite{Chu}, \cite{Co}, and \cite{BJP}) because they tend to be easier
to study.  
However the precise
connection between these systems and more realistic systems with
non-trivial boundary conditions has always been somewhat unclear.  Our
current program differs from earlier work in two crucial ways. The first
is that we make precise what the impact of our analysis is for the
(more realistic) systems on the line : namely in this paper we derive an 
expression for the velocity with which disturbances propagate in systems 
with periodic boundary, and in \cite{CV2} 
we numerically verify that this holds on the line as well. The second is that
we consider all possible nearest neighbor interactions: we do not
impose symmetries. This turns out to be of the utmost importance: when
we apply these ideas in \cite{CV2} it turns out that the systems with
the best performance are asymmetric.  Asymmetric systems (though not
the same as ours) have also been considered by \cite{BMH} and with
similar results.  However their methods are perturbative, and spectral
based. 
 In \cite{flock67} and
\cite{LFJ} asymmetric interactions are also studied, and it was shown
that in certain cases they may lead to exponential growth (in $N$) in
the perturbation.  In the later of these, the model is qualitatively
different because absolute velocity feedback is assumed (their method
is also perturbative and not global).  Signal velocities were employed
in earlier calculations namely \cite{WW} and \cite{MN}.  These
calculations have in common that they were done for
\emph{car-following} models. We are interested in a more general
framework, namely where automated pilots may pay attention \emph{also}
to their neighbor \emph{behind} them or indeed other cars further
afield.

Our model is \emph{strictly decentralized}. There are two reasons to
do that.  First, in high speed, high/density traffic, small
differences in measured absolute velocity may render that measurement
useless, if not dangerous, for the feedback.  Secondly, the desired
velocity, even on the highway, may not be constant. It will depend on
weather, time of day, condition of the road, and so on.  For these
reasons we limit ourselves to strictly \emph{decentralized} models
that only use information relative to the observers in the cars (see
\cite{flocks2} and \cite{position}).  Many authors study models
featuring a term proportional to velocity minus desired velocity (see
e.g. \cite{BJP}, \cite{BMH}, \cite{HM1}, \cite{Chu}, \cite{Co},
\cite{HM2}, and \cite{LFJ}).

\section{Flocking Model}
 \label{chap:definitions}


 We consider a model of a \emph{decentralized} flock of $N$ moving
 agents (e.g. cars), where each agent's acceleration depends linearly
 on on the differences between its own relative position and velocity,
 and those of some subset of neighbors. Letting $x_k$ be the position
 of the $k$\th agent, and $h_k$ its desired distance within the flock
 (typically $k$ times a fixed spacing $\delta$), the general linear
 decentralized flock satisfies
 \[
 \ddot x_k = \sum_{j\in \mathcal{N}_k} p_{j,k}( (x_k-h_k) - (x_j-h_j)) +
 v_{j,k}(\dot x_k -\dot x_j)
\]
where $\mathcal{N}_k$ is the set of neighbors for agent $k$, and $p_{j,k}$ and $v_{j,k}$ are 
the coefficients for how the difference of positions and velocities respectively between agent $k$ and $j$ 
affect the acceleration of agent $k$. The above
model is more general than that considered in this work, we restrict
ourselves to a leaderless decentralized flock with identical agents
and periodic boundary. These restrictions imply $p_{k,j}$ and
$v_{k,j}$ depend only on $j-k \mod N$, and that the neighborhood sets
$\mathcal{N}_k$ be shift invariant, e.g. $\mathcal{N}_{j+k} = \{
(j+i)\mod N : i \in \mathcal{N}_k \}$. We will also restrict ourselves to nearest neighbor systems.
To further simplify the resulting equations, we introduce the change of variables $z_k \equiv
x_k-h_k$ (see \cite{flocks2} for more details). We also introduce
constants $g_x$ and $g_v$, define $\rho_{x,j} = \tfrac{1}{g_x}
p_{j,0}$ for $j\neq 0$ and $\rho_{x,0}=\tfrac{1}{g_x} \sum_{j\neq 0}
p_{j,0}$ where all indices are treated mod $N$, and define
$\rho_{v,j}$ similarly. It will be convenient to allow negative
indices for $\rho_{x,j}$ by setting $\rho_{x,j+N}=\rho_{x,j}$,
similarly for $\rho_{v,j}$. In this notation, the flock equations
become the following :

\begin{definition} 
\label{defn:normalized system}
The system $S_N^*$ is given by the equation
\begin{equation}\label{eqn:normalized system}
\ddot z_k = g_x \sum_{j\in \mathcal{N}} \rho_{x,j} z_{k+j} + g_v
\sum_{j\in \mathcal{N}} \rho_{v,j}\dot z_{k+j} \equiv g_x\sum_{j=1}^N
L_{x,k,j} z_j + g_v\sum_{j=1}^N L_{v,k,j}\dot z_j
\end{equation}
where $\mathcal{N}=\{-1,0,1\}.$
%
The $N\times N$ matrices $L_x$ and $L_v$ defined implicitly above are
circulant matrices, as $L_{x,k,j}$ and $L_{v,k,j}$ depend only on $j-k
\mod N$. They also have row sums equal to 0, as the decentralized
condition has implied that 
\begin{equation} \label{eq:decentralized}
\sum_{j\in \mathcal{N}} \rho_{x,j} = \sum_{j\in\mathcal{N}} \rho_{v,j} = 0 . 
\end{equation}
We will accordingly refer to $L_x$ and $L_v$ as Laplacian matrices.
\end{definition}

{\noindent \bf Remark :} It is well known that circulant matrices have
orthogonal eigenbases, and are diagonalized by the discrete Fourier
transform (see \cite{KS}). This is the reason periodic boundary
conditions are so convenient.

It will be useful to write the equations of $S_N^*$ as a first order system:
\begin{equation}
\frac{d}{dt}\begin{pmatrix} z\\ \dot z \end{pmatrix}  =
M_N \begin{pmatrix} z\\ \dot z \end{pmatrix} \equiv
\begin{pmatrix} 0 & I \\ g_xL_x & g_vL_v \end{pmatrix} \begin{pmatrix} z\\ \dot z \end{pmatrix}
\label{eqn:first-order}
\end{equation}
This system has a 2-dimensional family of coherent solutions, namely:
\begin{equation*}
\forall \; i \quad z_i(t) = v_0t+x_0
\end{equation*}
where $v_0$ and $x_0$ are arbitrary elements of $\R$. These correspond
to the generalized eigenspace of $M_N$ for the eigenvalue 0.  It is easy to see
that all solutions converge to one of these coherent solutions if and
only if all other eigenvalues of $M_N$ have negative real part.  With
a slight abuse of notation we will call this case asymptotically
stable (see \cite{Ro} for precise definitions):

\begin{definition} The system in Equation \ref{eqn:first-order} is called
  asymptotically stable if it has a single eigenvalue equal to 0 with
  algebraic multiplicity 2, and all other eigenvalues have strictly
  negative real parts.
\label{defn:asympt}
\end{definition}

The discrete Fourier transform will play a fundamental role in our
analysis. We define  $\lambda_{x,m}$ and $\lambda_{x,m}$ as follows:
denote $\theta\equiv \frac{2\pi}{N}$ and set
\begin{equation}
\lambda_{x,m} \equiv g_x \sum_{j\in\mathcal{N}} \rho_{x,j}\, e^{ijm\theta} \quad \logand \quad
\lambda_{v,m} \equiv g_v \sum_{j\in\mathcal{N}} \rho_{v,j}\, e^{ijm\theta}
\label{eqn:the lambdas}
\end{equation}
%
Denote the vector $w_m$ by:
\begin{equation*}
 w_m \equiv \dfrac{1}{\sqrt{N}}\left(1, e^{im\theta}, e^{2im\theta}, \cdots e^{(N-1)im\theta} \right)^T
\end{equation*}
We furthermore define the moments of $g_x\rho_x$ and $g_v\rho_v$:
\begin{equation*}
I_{x\ell}\equiv g_x \sum_{j\in \mathcal{N}}\,\rho_{x,j}\, j^\ell \quad \logand \quad
I_{v\ell}\equiv g_v \sum_{j \in\mathcal{N}}\,\rho_{v,j}\, j^\ell
\end{equation*}
and observe that $\lambda_{x,m}$ can be expanded as 
\begin{equation}
\lambda_{x,m}  
               =  im\theta\; I_{x,1} - \dfrac{m^2\theta^2}{2}\;I_{x,2}-i\; \dfrac{m^3\theta^3}{3!}\;I_{x,3}
               +\dfrac{m^4\theta^4}{4!}\;I_{x,4}+i\;\dfrac{m^5\theta^5}{5!}\;I_{x,5} \cdots
\label{eq:lambda_expansion_in_moments}
\end{equation}
An analogous expansion for $\lambda_{v,m}$ can also be given.

\section{Asymptotic Stability}
 \label{chap:stability}

In this section we state and prove necessary and sufficient conditions for
nearest neighbor systems to be asymptotically stable.

\begin{proposition}
Let $L_x$ and $L_v$ be the Laplacians defined in
Definition \ref{defn:normalized system}. The eigenvalues of $g_xL_x$ are $\lambda_{x,m}$ with
associated eigenvector $w_m$ (where $m\in\{0,\cdots,N-1\}$). Similarly, $\lambda_{v,m}$ and $w_m$ form eigenpairs for $g_vL_{v}$.
\label{prop:evals Lapl}
\end{proposition}
\begin{proof}
This follows immediately from the previous remark as $L_x$ and $L_v$ are circulant matrices.
\end{proof}

{\noindent \bf Remark:} Even though $L_x$ and $L_v$ have bases of orthogonal
eigenvectors, $M_N$ does not. Instead, the eigenvectors of the $2N
\times 2N$ matrix $M_N$ lie within $N$ two-dimensional
subspaces which are orthogonal to each other. Each of these may be
spanned by two not necessarily orthogonal eigenvectors, or by an
eigenvector and a (Jordan) generalized eigenvector.  This is made
precise below:

\begin{proposition}
The eigenvalues $\nu_{m\pm}$ ($m\in\{0,\cdots N-1\}$) of $M$ are given by the solutions of
\begin{equation*}
\nu^2 -\lambda_{v,m}\nu -\lambda_{x,m}=0\quad \Rightarrow \quad
\nu_{m\pm} = \dfrac{\lambda_{v,m}}{2}\pm \sqrt{\dfrac{\lambda_{v,m}^2}{4}+\lambda_{x,m}}
\end{equation*}
with associated eigenvectors given by $\begin{pmatrix}
  w_m\\\nu_{m\pm} w_m \end{pmatrix}$.
\label{prop:2}
\end{proposition}
\begin{proof}
Let  $\nu$ be an eigenvalue of $M_N$, with eigenvector
written as $\begin{pmatrix} q\\ u \end{pmatrix}$. Then
\begin{equation*}
\begin{pmatrix} 0 & I \\ g_xL_x & g_vL_v \end{pmatrix} 
\begin{pmatrix}  q \\ u \end{pmatrix} 
= \nu \begin{pmatrix}  q \\ u \end{pmatrix},
\end{equation*}
which implies first that $u=\nu q$ and then that $(g_xL_x+g_vL_v \nu )
q = \nu^2 q$. The latter shows that $\nu^2$ is an eigenvalue of the
circulant matrix $g_xL_x + \nu g_v L_v$, which from Proposition
\ref{prop:evals Lapl} has eigenvalues given by $\lambda_{x,m} + \nu
\lambda_{v,m}$, for $0\leq m \leq {N-1}$. This implies $\nu$ satisfies
 $\nu^2=\lambda_{x,m}+\nu \lambda_{v,m}$ for some $m$. Finally, letting
$\nu_{m\pm}$ be as above, it is straightforward to show
$\begin{pmatrix} w_m \\ \nu_{m\pm} w_m\end{pmatrix}$ are
eigenvectors with eigenvalue $\nu_{m\pm}$.
\end{proof}

\begin{figure}
\center
\includegraphics[height=2.5in]{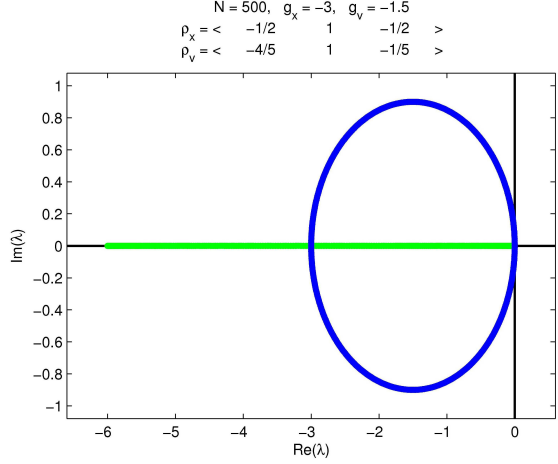}   
\caption{A representative figure for the calculation of the
  eigenvalues for 500 agents. The values of the parameters are given
  in the figures. Green line: $\lambda_{x,m}$, Blue ellipse: $\lambda_{v,m}$. 
  \label{fig:phasevelocities1}}
\end{figure}

\begin{definition}
Define $\lambda_x:S^1\rightarrow \C$ and $\lambda_v:S^1\rightarrow \C$ by
\begin{equation*}
\lambda_x(\phi) \equiv g_x \sum_{j\in\mathcal{N}}\,\rho_{x,j}\, e^{ij\phi} \logand
\lambda_v(\phi) \equiv g_v \sum_{j\in\mathcal{N}}\,\rho_{v,j}\, e^{ij\phi},
\end{equation*}
and define the curve $\gamma \subset \C$ to be the set of all $\nu \in \C$ satisfying
\begin{equation*}
\nu^2 -\lambda_v(\phi)\nu -\lambda_x(\phi)=0
\label{defn:eigencurve}
\end{equation*}
for some $\phi\in[0,2\pi]$.
\end{definition}

Note first that $\lambda_x$ and $\lambda_v$ are be independent of $N$
(provided $N$ is larger than half the size of the neighborhood
$\mathcal{N}$). This implies in turn that $\gamma$ depends on all
parameters of the flocking model except $N$.  As $N$ grows and the
remaining parameters are fixed, the eigenvalues of the associated
Laplacians tend to fill out curves denoted by $\lambda_x$ and
$\lambda_v$.  This is illustrated in Figure
\ref{fig:phasevelocities1}. The same holds for the eigenvalues of
$M_N$ (see Figure \ref{fig:phasevelocities2}), in particular :

\begin{corollary} The set $\spec(M_N)$ of eigenvalues of $M_N$ satisfies:
\begin{equation*}
\spec(M_N) \subset \gamma \quad \logand \quad \lim_{N\rightarrow\infty} \spec(M_N) = \gamma
\end{equation*}
where the second limit is taken relative to the Haussforff
metric $d_H$. Similar results hold for the spectra of $L_x$ and $L_v$ and
the images of the curves $\lambda_x$ and $\lambda_v$.
\label{cory:eigensets}
\end{corollary}
\begin{proof}
  For any $\nu_{m\pm}$, setting $\phi=m\theta$ in Definition
  \ref{defn:eigencurve} gives $\lambda_v(\phi) = \lambda_{v,m}$ and
  $\lambda_x(\phi) = \lambda_{x,m}$ which shows $\nu_{m\pm} \in
  \gamma$. 

  The second assertion is equivalent to showing $\lim_{N\to\infty}
  d_H(\gamma,M_N)=0$.  Let $g$ be the set-valued function from
  $[0,2\pi]$ given by $g(\phi) = \{\nu : \nu^2 - \lambda_{v}(\phi) v -
  \lambda_x(\phi)=0\}$, so that $\gamma = \cup_{\phi\in[0,2\pi]}
  g(\phi)$. As $\lambda_v$ and $\lambda_x$ are continuous, and roots of
  polynomials depend continuously on their coefficients, $g$ is
  continuous. As its domain is compact, it is uniformly
  continuous. Now fix $\epsilon>0$. For any $\nu \in \gamma$, $\nu\in
  g(\phi)$ for some $\phi$, there is $\delta>0$ (independent of
  $\phi$) so that $|\phi'-\phi| < \delta \implies
  d_H(g(\phi'),g(\phi))<\epsilon$. We may take $N$ large enough so
  that $|\tfrac{k}{2\pi N}-\phi|<\delta$ for some $k$. As
  $g(\tfrac{k}{2\pi N})\subset \spec(M_N)$, this implies there are
  points in $\spec(M_N)$ that are distance less than $\epsilon$ from
  $\nu$. As $\nu$ was arbitrary, and we already have
  $\spec(M_N)\subset \gamma$, this implies $d_H(\gamma,\spec(M_N)) <
  \epsilon$, which suffices to prove the desired limit.
\end{proof}

\begin{figure}
\center
\includegraphics[height=2.5in]{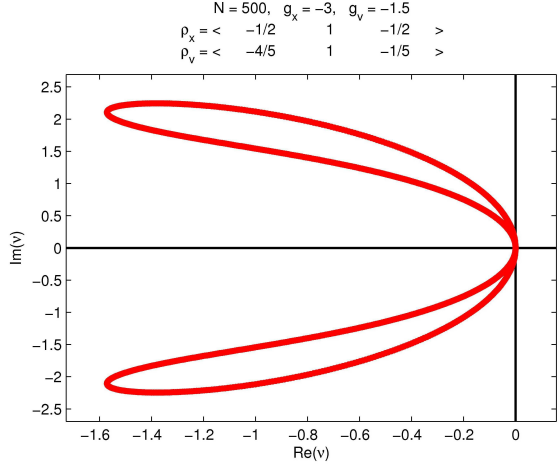}   
\caption{ A representative figure for the calculation of the
  eigenvalues for 500 agents.  The values of the parameters are given
  in the figures. The eigenvalues of $M_N$ of Proposition
  \ref{prop:2}.
\label{fig:phasevelocities2}}
\end{figure}

We next present a necessary condition for asymptotic stability.
\begin{proposition} 
  \label{prop:Ix1=0}
  If $I_{x,1}\neq 0$, then for sufficiently large $N$ the system
  $S^*_N$ in Definition \ref{defn:normalized system} is not asymptotically
  stable.
\end{proposition}
\begin{proof}

We will first show that for large enough $N$ and small enough $m$, the eigenvalues $\nu_{m\pm}$ 
are close to $\pm \sqrt{\lambda_x(\tfrac{2\pi m}{N})}$, which will have positive real part for one branch.

  We may write the expansions $\lambda_{x}(\phi) = i\phi I_{x,1}
  -\tfrac{\phi^2}{2}I_{x,2} - ...$ and $\lambda_{v}(\phi) = i\phi
  I_{v,1} -\tfrac{\phi^2}{2}I_{v,2} - ...$ similar to equation
  (\ref{eq:lambda_expansion_in_moments}). If $I_{x,1}\neq 0$, these
  imply existence of constants $A$, $B$, and $\delta_1>0$ so that
  $|\phi|<\delta_1$ implies both $|\lambda_{x}(\phi)| \geq B|\phi|$ and
  $|\lambda_{v}(\phi)| \leq A|\phi|$. Set $\nu_{\pm}(\phi) =
  \tfrac{\lambda_v(\phi)}{2} \pm \sqrt{\tfrac{\lambda_v^2(\phi)}{4} + \lambda_x(\phi)}$
(so that $\nu_{m\pm}=\nu_\pm(\tfrac{2\pi m}{N})$, as in Proposition \ref{prop:2}).
  By taking $|\phi| < \delta_2$ we may ensure $|\sqrt{\lambda_x +
    \lambda_v^2/4}-\sqrt{\lambda_x}| < \tfrac{1}{4}
  |\sqrt{\lambda_x}|$. Then $|\nu_{\pm}(\phi) -
  \pm\sqrt{\lambda_x(\phi)}|=|\lambda_v(\phi)/2\pm(\sqrt{\lambda_v^2/4+\lambda_x}-\sqrt{\lambda_x})|
  \leq \tfrac{1}{4}|\lambda_x| + \tfrac{1}{2}A|\phi|$. As
  $|\sqrt{\lambda_x}|\geq \sqrt{B |\phi|}$, by taking
  $|\phi|<\delta_3$ we can ensure $|\nu_{\pm}(\phi) -
  \pm\sqrt{\lambda_x(\phi)}| \leq \tfrac{1}{2}|\sqrt{\lambda_x}|$.

  For $|\phi|<\delta_4$ we can similarly derive the estimate
  $|\pm\sqrt{\lambda_x(\phi)} - \pm\sqrt{i\phi I_{x,1}}| \leq
  \sqrt{i\phi I_{x,1}} C |\phi|$ for some constant $C$. Now as
  $\pm\sqrt{i\phi I_{x,1}} = \pm \sqrt{\phi I_{x,1}}
  (\tfrac{\sqrt{2}}{2} + i \tfrac{\sqrt{2}}{2})$, at least one of
  these branches must have positive real part. The previous estimates
  imply that for $|\phi|<\min(\delta_1,\delta_2,\delta_3,\delta_4)$, one of the branches of
  $\nu_{\pm}(\phi)$ must have positive real part. Thus for sufficiently
  large $N$, there are eigenvalues of $M_N$ with positive real part, so $S_N^*$ is not
  asymptotically stable.
\end{proof}

In the remainder of this section we use a global method to determine a better
condition for asymptotic stability for nearest neighbor systems.

\begin{proposition} The system $S_N^*$ of Definition \ref{defn:normalized system} is asymptotically stable for
all $N$ (all other parameters fixed) if $\rho_{x,j}$ is symmetric, and 
\begin{equation*}
\forall \; \phi\neq 0 \; : \; \Re(\lambda_x(\phi))<0 \logand
\Re(\lambda_v(\phi))<0 .
\end{equation*}
Instability will occur for large $N$ if either of the opposite inequalities holds for some
$\phi \neq 0$.
\label{prop:Routh}
\end{proposition}
\begin{proof}
  By the Routh-Hurwitz criterion applied to the equation
  $\nu^2-\lambda_{v,m}\nu -\lambda_{x,m}=0$ with complex coefficients
  (see \cite{Laff}), we see that all nonzero eigenvalues given in
  Proposition \ref{prop:2} all have negative real parts if and only if
  for all $m\in \{1,\cdots N-1\}$ we have (in the notation of Equation
  \ref{eq:lambda_expansion_in_moments}):
\begin{eqnarray*}
\Re(\lambda_{v,m}) &<& 0\\
2 \Re(\lambda_{x,m}) &<&|\lambda_{v,m}|^2 \\
\Re(\lambda_{x,m}) \Re(\lambda_{v,m}) + \Im(\lambda_{x,m}) \Im(\lambda_{v,m}) &>& 0 \\
\Re(\lambda_{x,m}) [\Re(\lambda_{v,m})]^2 + \Re(\lambda_{v,m})
\Im(\lambda_{x,m}) \Im(\lambda_{v,m}) + [\Im(\lambda_{x,m})]^2 &<& 0 
\end{eqnarray*}
If $\rho_{x,j}$ are symmetric then $\Im(\lambda_{x,m})$ is zero for
all $m$. In this case the Routh-Hurwitz conditions reduce to: for
all $m\in \{1,\cdots N-1\}$, $\Re(\lambda_{x,m})<0$ and
$\Re(\lambda_{v,m})<0$.  As $\lambda_{x,m}=\lambda_x(\tfrac{2\pi
  m}{N})$ and $\lambda_{v,m}=\lambda_v(\tfrac{2\pi m}{N})$, stability
for all $N$ follows if $\Re(\lambda_x(\phi))<0 \logand
\Re(\lambda_v(\phi))<0$ for all $\phi \neq 0$ (the case of $\phi=0$
is excluded because the zero eigenvalue is excluded from our
definition of asymptotic stability ).

On the other hand, if either $\Re(\lambda_x(\phi))>0$  or
$\Re(\lambda_v(\phi))>0$  for some $\phi \neq 0$, then as the set of points
$\{\tfrac{2\pi m}{N}\}$ for all $N>0$ and $1\leq m \leq N$ is dense in
$[0, 2\pi]$, and the functions $\lambda_x$ and $\lambda_v$ are
continuous, then there must be some $N$ and $m$ so that either
$\Re(\lambda_{x,m})>0$ or $\Re(\lambda_{v,m})>0$, in which case
$S_N^*$ is not asymptotically stable. \end{proof}

We are now in a position to state and prove the main theorem of this
section. Recall that we identify $\rho_{x,-1}$ and $\rho_{v,-1}$ with
$\rho_{x,N-1}$ and $\rho_{v,N-1}$ in Definition \ref{defn:normalized
  system}.

\begin{theorem} Suppose $S_N^*$ is as defined in Definition
  \ref{defn:normalized system}, with $\mathcal{N}=\{-1,0,1\}$. Then
$S_N^*$ is asymptotically stable for all $N$ if and only if
$\rho_{x,-1}=\rho_{x,1}$, $g_x \rho_{x,0}< 0$, and $g_v \rho_{v,0} < 0$.
\label{theo:main1}
\end{theorem}
\begin{proof}
  Let $S_N^*$ be asymptotically stable for all $N$. First, Proposition
  \ref{prop:Ix1=0} implies $I_{x,1}=0$, which for
  $\mathcal{N}=\{-1,0,1\}$ implies $\rho_{x,-1}=\rho_{x,1}$.  Next,
  Equation \ref{eq:decentralized} implies that $\rho_{x,0} = -
  (\rho_{x,1}+\rho_{x,-1})$, which implies
\begin{align*}
\Re(\lambda_x(\phi))&=\Re(g_x(e^{-i\phi} \rho_{x,-1} + \rho_{x,0} +
e^{i\phi} \rho_{x,1}))\\
&= g_x \rho_{x,0}( 1-\cos(\phi))
\end{align*}
Similarly $\Re(\lambda_v)=g_v \rho_{v,0}( 1-\cos(\phi))$. As
$(1-\cos(\theta)) > 0$ for $\phi \neq 0$, Proposition \ref{prop:Routh}
implies we must have $g_v \rho_{v,0}<0$ and $g_x \rho_{x,0}<0$.


To prove the other direction, let $\rho_{x,-1}=\rho_{x,1}$, $g_x
\rho_{x,0}<0$ and $g_v \rho_{v,0}<0$. The same calculation as above
shows $\Re(\lambda_x(\phi))<0$ and $\Re(\lambda_v(\phi))<0$ for
$\phi\neq 0$, then Proposition \ref{prop:Routh} implies $S_N^*$ is
asymptotically stable.
\end{proof}

\section{Signal Velocities} \label{chap:signal}

The main result of this section is the determination of the signal
velocity in asymptotically stable systems as characterized in Theorem
\ref{theo:main1}.  The signal velocity is the velocity with which
disturbances (such as a short pulse) propagate
through the flock. In general, signal velocities in dispersive media
may be difficult to determine.  The reason is that a pulse consists of
a superposition of plane waves, typically each with a different phase
velocity. 
If the component plane waves have different phase velocities, the
pulse may spread out over time (dispersion), and the determination of
arrival time of the signal may becomes problematic. For details we
refer to \cite{Brillouin}.

For nearest neighbor systems of Definition \ref{defn:normalized system} we define:
\begin{equation*}
a\equiv \dfrac{I_{v,1}^2}{4} +\dfrac{I_{x,2}}{2}=
\dfrac{(\rho_{v,0}+2\rho_{v,1})^2g_v^2}{4}+\dfrac{-g_x\rho_{x,0}}{2}
\end{equation*}

{\noindent \bf Remark:} From now on we will restrict our attention to (stable) systems satisfying the
conditions of Definition \ref{defn:normalized system} and the conclusions of Theorem \ref{theo:main1}.
Note that for these systems $a>0$. In order to simplify notation we will also
(without loss of generality, because of Theorem \ref{theo:main1})
re-scale $g_x$ and $g_v$ so that the values of $\rho_{x,0}$ and $\rho_{v,0}$ are 1 from now on.

{\noindent \bf Remark:} From the definitions it is clear that $\nu_{N-m,\pm}$ can be identified
with $\nu_{-m,\pm}$ and that $\nu_{-m,\pm}$ is the complex conjugate of $\nu_{m\pm}$.
It will be convenient in this section to relabel these eigenvalues so that $m$ runs
from $\lceil-(N-1)/2 \rceil$ to $\lceil (N-1)/2 \rceil$. For simplicity of notation,
we will however write $\sum_{m=\lceil-(N-1)/2 \rceil}^{\lceil (N-1)/2 \rceil}$ as
$\sum_{m=-N/2}^{N/2}$.

\begin{proposition}
Let $S_N^*$ as in Definition \ref{defn:normalized system} and Theorem \ref{theo:main1}.
Then $a>0$ and the eigenvalues $\nu_{m\varepsilon}$ of $M_N$ can be expanded as
(with $\varepsilon=\pm 1$ and $\theta\equiv \frac{2\pi}{N}$):
\begin{eqnarray*}
\nu_{m\varepsilon} = & im\theta & \left(\dfrac{I_{v,1}}{2}+\varepsilon a^{1/2}\right)+\\
& m^2\theta^2 & \left( -\dfrac{I_{v,2}}{4} -\varepsilon \dfrac{\left(\frac{I_{v,1}I_{v,2}}{4} + \frac{I_{x,3}}{6}\right)}{2a^{1/2}}\right)+\\
& im^3\theta^3 &\left( -\dfrac{I_{v,3}}{12} -
\varepsilon\dfrac{\left(\frac{I_{v,1}I_{v,3}}{12} + \frac{I_{x,4}}{24} + \frac{I_{v,2}^2}{16} \right)}{2a^{1/2}} +\varepsilon \dfrac{\left(\frac{I_{v,1}I_{v,2}}{4} + \frac{I_{x,3}}{6} \right)^2}{8a^{3/2}}\right)+\\
& m^4\theta^4 & \left( \dfrac{I_{v,4}}{48} + \varepsilon\dfrac{\left(\frac{I_{v,2}I_{v,3}}{24} + \frac{I_{v,2}I_{v,3}}{48} + \frac{I_{x,5}}{125} \right)}{2a^{1/2}} -\varepsilon\dfrac{\left(\frac{I_{v,1}I_{v,2}}{4} + \frac{I_{x,3}}{6} \right)
 \left(\frac{I_{v,1}I_{v,3}}{12}+\frac{I_{x,4}}{24} + \frac{I_{v,2}^2}{16}\right)}{4a^{3/2}} \right. \\
 && \left. +\varepsilon \dfrac{\left(\frac{I_{v,1}I_{v,2}}{4} + \frac{I_{x,3}}{6} \right)^3}{16a^{5/2}} \right)+\cdots
\end{eqnarray*}
\label{prop:expansion eigenvals}
\end{proposition}

\begin{proof}
Expand $\nu_{m\pm}$ given in Proposition \ref{prop:2} in powers of $\theta$ using
\begin{equation*}
a\neq 0 \Rightarrow \sqrt{z-a} = \pm i\sqrt{a}\; \left( 1 - \dfrac{z}{2a}-\dfrac{z^2}{8a^2}- \dfrac{z^3}{16a^3}\cdots\right)
\end{equation*}
After a substantial but straightforward calculation the result is
obtained.      
\end{proof}

The phase velocity of the time-varying sinusoid $f(x,t)=e^{i(\omega t-
  {\mathrm k} x)}$ on the real line is defined by the evolution of points
of constant phase : $\omega t - {\mathrm k} x(t) = c$, which gives the
phase velocity $\omega /{\mathrm k}$. 


Disturbances in the positions of agents in the flock may be decomposed
in terms of solutions to Equation \ref{eqn:normalized system} which
are damped sinusoidal waves as functions of time and agent number. We
define phase velocity in units of number of agents per unit time, as
follows.

\begin{definition} \label{def:phase_velocity}
The set of solutions 
\begin{equation}
z_k(t)= A e^{i(\omega t - b k)} e^{-a  t}
\end{equation}
has phase velocity $\omega / b$.
\end{definition}
On our way to studying the propagation velocity of disturbances in the
system $S_N^*$, we will characterize its phase velocities. We first
establish the following :

\begin{lemma}
\label{lemma:nu_m_opposite_signs}
For $S_N^*$ as in Theorem \ref{theo:main1}, the imaginary parts of the eigenvalues $\nu_{m\pm}$
have opposite signs for $m\neq 0$.
\end{lemma}
\begin{proof}
Set $\nu_{m+} = \alpha_1+i\beta_1$ and $\nu_{m-} =
\alpha_2+i\beta_2$. As $\nu_{m\pm}$ are roots of
$\nu^2-\nu\lambda_{v,m} - \lambda_{x,m}$, and 
\begin{equation*}
(\nu-\mu_1)(\nu-\mu_2)=
\nu^2-(\alpha_1+\alpha_2+i(\beta_1+\beta_2))\nu+\alpha_1\alpha_2-\beta_1\beta_2+
i(\alpha_1\beta_2+\alpha_2\beta_1) ,
\end{equation*}
we can identify $\lambda_{x,m} = -(\alpha_1\alpha_2-\beta_1\beta_2) + i
(\alpha_1\beta_2+ \alpha_2\beta_1)$. We have $\Im(\lambda_{x,m})=0$ as
$\rho_{x,j}$ is symmetric, so
$\alpha_1\beta_2+\alpha_2\beta_1=0$. Solving gives
$\beta_1=-\tfrac{\alpha_1}{\alpha_2}\beta_2$. But $\alpha_1 <0$ and
$\alpha_2<0$ because $S_N^*$ is asymptotically stable, 
so $\beta_1$ and $\beta_2$ have opposite signs.
\end{proof}

\begin{lemma}
For $S_N^*$ as in Theorem \ref{theo:main1}, phase velocities are given by
\begin{equation} \label{eq:phase_velocity_c}
c_{m+}=\frac{-\Im(\nu_{m-})}{m\theta} > 0 \quad\logand \quad
c_{m-}=\frac{-\Im(\nu_{m+})}{m\theta} < 0
\end{equation}
for $1\leq m\leq \frac{N}{2}$. 
\label{lem:phasevelocity}
\end{lemma}
\begin{proof}
Lemma \ref{lemma:nu_m_opposite_signs} implies $\Im(\nu_{m+})$ and
$\Im(\nu_{m-})$ have opposite signs. Redefine (if necessary, see
Proposition \ref{prop:2}) the subscripts ``$+$" and ``$-$" 
so that $\nu_{m+}$ has positive imaginary part, and $\nu_{m-}$ has
negative imaginary part.

We now derive phase velocities where $+$ denotes going from agent ``0"
towards agent ``N".  The expression for the $k\th$ entry of the
time-evolution of the solution corresponding to the eigenvalue
$\nu_{m\pm}$ is (see Proposition \ref{prop:evals Lapl}) is (up to an
arbitrary multiplicative constant) :
\begin{equation*}
z_k= e^{(\nu_{m\pm}) t}e^{ i k m  \theta}
=e^{\Re(\nu_{m\pm})\,t} e^{i(\Im(\nu_{m\pm}) t + k m \theta)}
=e^{\Re(\nu_{m\pm})\,t} e^{i(\Im(\nu_{m\pm}) t - (-m \theta)k)}
\end{equation*}
Comparing this to Definition \ref{def:phase_velocity} shows these two
solutions have phase velocities $c_{m,+}$ and $c_{m,-}$ as given in
\ref{eq:phase_velocity_c}. 
\end{proof}

From Proposition \ref{prop:expansion eigenvals} we see that the
eigenvalues close to the origin form four branches which intersect at
the origin. Namely $\varepsilon$ can be +1 or -1, and the counter $m$
can be positive or negative. This is illustrated in Figure
\ref{fig:phasevelocities2}. So for given $|m|$ we get two phase
velocities: one in each direction.

\begin{lemma}
  For $S_N^*$ as in Theorem \ref{theo:main1}, the phase velocities
  $c_{m\varepsilon}$ of Lemma \ref{lem:phasevelocity} can be expanded
  as ($\varepsilon\in \{-1,1\}$):
\begin{eqnarray*}
c_{m\varepsilon} &=& -\dfrac{g_v(1+2\rho_{v,1})}{2}+\varepsilon
\sqrt{\dfrac{g_v^2(1+2\rho_{v,1})^2}{4}-\dfrac{g_x}{2}} + \\
&& m^2\theta^2\left(\dfrac{g_v(1+2\rho_{v,1})}{12}
  -\varepsilon\;\dfrac{2g_v^2(1+2\rho_{v,1})^2-g_x+\frac32
    g_v^2}{24[g_v^2(1+2\rho_{v,1})^2-2g_x]^{1/2}}
  +\varepsilon\;\dfrac{g_v^2(1+2\rho_{v,1})}{16[g_v^2(1+2\rho_{v,1})^2-2g_x]^{3/2}}
\right)\\
&& + \bigoh((m\theta)^4) \\
\end{eqnarray*}
The real parts of the associated eigenvalues can be expanded as:
\begin{equation*}
\Re(\nu_{m\varepsilon})= m^2 \theta^2\left(\dfrac{g_v}{4}+ \varepsilon\;
  \dfrac{g_v^2(1+2\rho_{v,1})}{4[g_v^2(1+2\rho_{v,1})^2-2g_x]^{1/2}} \right)+\bigoh((m\theta)^4)
\end{equation*}
\label{lem:phasevelocity2}
\end{lemma}

\noindent

\begin{proof}
With the reduction $\rho_{x,0}=\rho_{v,0}=1$ as described in the remark in at the 
beginning of section \ref{chap:signal}, Theorem \ref{theo:main1}
implies $\rho_{x,1}=\rho_{x,-1}=-\tfrac{1}{2}$, and Equation
(\ref{eq:decentralized}) implies $\rho_{v,-1} = -(1+\rho_{v,1})$. We can then
compute all of the moments 
\begin{align*}
I_{x,j} &= (-1)^j(-\tfrac{1}{2}) + 1^j(-\tfrac{1}{2}) &= \ & \begin{cases}
0 \mbox{\quad $j$ even}\\
1 \mbox{\quad $j$ odd}
\end{cases} \\
I_{v,j} &= (-1)^j(-(1+\rho_{v,1}))+1^j \rho_{v,1} &= \ &\begin{cases}
-1\quad \quad \ \;\mbox{\quad  $j$ even} \\
1+2\rho_{v,j} \mbox{\quad  $j$ odd}
\end{cases}
\end{align*}

Substituting the expansion from Proposition \ref{prop:expansion
  eigenvals} into the expressions for the phase velocity
${c_{m\varepsilon}=-\tfrac{\Im(\nu_{m-\epsilon})}{m\theta}}$ from
Lemma \ref{lem:phasevelocity}, and using the above expressions for the
moments $I_{x,j}$ and $I_{v,j}$ gives the desired expansion.
\end{proof}

For any set of initial conditions $z_k(0)$ and $\dot z_k(0)$, there
are unique constants $a_m$ and $b_m$ so that the solution of the
system $S_N^*$ has the form

\begin{equation} \label{eq:zkt_solution_expansion}
z_k(t)= \sum_{m=-N/2}^{N/2}\,a_m e^{im\theta k}\;e^{\nu_{m+}t}+
\sum_{m=-N/2}^{N/2}\,b_m e^{im\theta k}\;e^{\nu_{m-}t}
\end{equation}

Our main result of this paper is to show that the first sum represents
a signal travelling to the left (decreasing agent number), that may be
approximated by a travelling wave with a single signal
velocity. Likewise, the second sum represents a signal travelling to
the right. We first need a small technical lemma.

\begin{lemma}\label{lemma:exp_ab}
There is a $\delta>0$ such that for all $a$,$b$ satisfying $|a|<\delta$ and $|b|<\delta$, it follows that 
$|e^{a}-e^{b}|<2|a-b|$.
\end{lemma}
\begin{proof}
$e^a-e^b = (a-b)+\sum_{n=2}^{\infty} \tfrac{1}{n!}(a^n-b^n)$. Using 
$a^n-b^n=(a-b)\sum_{j=0}^{n-1} a^j b^{n-1-j}$ we have
$|e^a - e^b|=|a-b|\left(1+\sum_{n=2}^{\infty} \tfrac{1}{n!}
\left(\sum_{j=0}^{n-1} a^j b^{n-1-j}\right)\right)$. The term
multiplying $|a-b|$ expression is a convergent power series, so is
continuous, and approaches 1 as $a\to 0$ and $b\to 0$. The desired
inequality then follows.
\end{proof}

We now address the first term in Equation \ref{eq:zkt_solution_expansion}.

\begin{proposition} \label{prop:timedomain1}
Let $S_N^*$ be as in Theorem \ref{theo:main1}, and $c_-\equiv
c_{0-}$ as given in Lemma \ref{lem:phasevelocity2} ($m=0$). 
Suppose the initial conditions are such that
$b_m=0$ for all $m$, in the expansion in Equation (\ref{eq:zkt_solution_expansion}). 
In addition, suppose that the coefficients $a_m$ satisfy $|a_m| < \tfrac{M}{m^p}$ for some $p>1$.
Fix $K>1$ and $0<\alpha<\beta<1$. Then, for all $t\in[\tfrac{N}{|c_-|}, K\tfrac{N}{|c_-|}]$, there is a function $f_-$ so that
\begin{align}\label{eq:prop_diff_bound}
|z_k(t) - f_-(k-c_-t)| < \frac{MDK}{|c_-|}N^{3\alpha - 1} 
&+ \tfrac{2M}{p-1}\left( (N^\alpha-1)^{1-p} - (N^\beta-1)^{1-p} \right) e^{-\mathcal{C}(\alpha,\beta) t}  \notag \\
&+ \tfrac{2M}{p-1}\left( (N^\beta-1)^{1-p} \right) e^{-\mathcal{C}(\beta,1) t} 
\end{align}
for sufficiently large $N$, where
$D$ is a constant, and
\begin{equation*}
\mathcal{C}_{}(a,b) = \min_{N^a \leq |m| \leq \min(N/2,N^b)} |\Re(\nu_{m+})|.
\end{equation*}
%
%
%
%
%
%
%
%
%
\end{proposition}
\begin{proof}
Consider the signal after time $t$:
\begin{equation}
\label{eq:zkt_split3}
z_k(t)= \left( \sum_{|m|<N^\alpha} +
\sum_{|m|\in[N^\alpha,N^\beta)} + \sum_{|m| \geq N^\beta}\right)\;a_m
e^{im\theta k}\, e^{\nu_{m+} t}
\end{equation}
Using Lemma \ref{lem:phasevelocity} we obtain
\begin{equation*}
e^{\nu_{m+} t }= e^{\Re(\nu_{m+})t} e^{i\Im(\nu_{m+})t} = e^{\Re(\nu_{m+})t}\; e^{-im\theta c_{m-}t}
\end{equation*}
Substituting back to Equation (\ref{eq:zkt_split3}) gives
\begin{equation}
z_k(t)= \left( \sum_{|m|<N^\alpha} +
\sum_{|m|\in[N^\alpha,N^\beta)} + \sum_{|m| \geq N^\beta}\right)
a_m e^{\Re(\nu_{m+})t} e^{i m \theta (k-c_{m-} t)}
\end{equation}

We note that this form shows that the solution is a sum of sinusoids,
with phase velocity $c_{m-}$, that are damped by the exponential
factor $e^{\Re(\nu_{m+})}$. Intuitively, our result follows because
for small $m$ the phase velocities are all close to the constant
$c_{-}$ and the damping is minimal, while for large $m$ the damping is
large enough so that we can ignore that the phase velocities depend on
$m$.

Set $f_-(z) = \sum_{|m|<N^\alpha} a_m e^{im\theta z }$. We then see
\begin{align} \label{eq:abs_zktmfm}
|z_k(t) - f_-(k-c_- t)| \leq & \left| \sum_{|m|<N^\alpha} a_m \left( e^{\Re(\nu_{m+}) t}
  e^{im\theta (k-c_{m-} t)} - e^{im\theta (k-c_-t)} \right)\right| + \notag \\
&\left|\sum_{|m|\in [N^\alpha,N^\beta]}  a_m e^{\Re(\nu_{m+}) t}
  e^{im\theta (k-c_{m-} t)} \right| +  \notag \\
  &\left| \sum_{|m|>N^\beta} a_m e^{\Re(\nu_{m+}) t}
  e^{im\theta (k-c_{m-} t)} \right|
\end{align}
We will bound these three sums separately. For the first sum, we may factor
\begin{equation*} e^{\Re(\nu_{m+}) t} e^{im\theta (k-c_{m-} t)} - e^{im\theta
  (k-c_-t)} = e^{im\theta(k-c_{m-} t)}\left(e^{\Re(\nu_{m+}) t} -
    e^{im\theta(c_{m-}-c_-) t} \right)
\end{equation*}
Applying lemma \ref{lemma:exp_ab} with $a=\Re(\nu_{m+}) t$ and
$b=im\theta(c_{m-}-c_-) t$ gives (provided these can be made
small enough)
\begin{align*}
|e^{im\theta(k-c_{m-} t)}\left(e^{\Re(\nu_{m+}) t} -
    e^{im\theta(c_{m-}-c_-) t} \right)| & \leq
2|\Re(\nu_{m+}) t  - im\theta(c_{m-}-c_-) t| \\
& \leq 2 |t| ( |\Re(\nu_{m+})| + |m\theta| |c_{m-} - c_-| )
\end{align*}
The expansions in Lemma \ref{lem:phasevelocity2} shows that both
$|c_{m-} - c_-|$ and $\Re(\nu_{m+})$ are $\bigoh(\theta^2 m^2)$. In
addition, for $|m|<N^\alpha$, $|m\theta| < 2\pi m N^{\alpha-1}$. This
implies there is a constant $C$ so that $|\Re(\nu_{m+})| + |m\theta|
|c_{m-} - c_-| < C (N^{\alpha-1})^2$ (and also justifies that $a$ and
$b$ as defined above may be made sufficiently small, by taking $N$
sufficiently large). The first sum in Equation \ref{eq:abs_zktmfm} has
$2N^\alpha$ terms, each has $|a_m|<M$; the entire sum is then bounded
by $2 N^\alpha M 2 |t| C (N^{\alpha-1})^2 = 4 M |t| C N^{3\alpha -
  2}$. For $t \in [N/|c_-|, KN/|c_-|]$, the same entire sum is bounded by
$4M K (N/|c_-|) C N^{3\alpha - 2} = \tfrac{M_- D K}{|c_-|} N^{3\alpha - 1}$
for $D=4 C$.

For the second sum, we have $|e^{\Re(\nu_{m+}) t} |<
  e^{-\mathcal{C}(\alpha,\beta) t}$. Using the decay condition on the coefficients $a_m$ 
shows the second sum is bounded by
\[
\left( 2 \sum_{m=N^\alpha}^{N^\beta} M m^{-p} \right) e^{-\mathcal{C}(\alpha,\beta)t}
\]
The elementary bound $\sum_{m=a}^b m^{-p}\leq \int_{a-1}^{b-1} x^{-p} dx = \tfrac{1}{p-1} ( (a-1)^{p-1} - (b-1)^{p-1})$,
applied above, gives the stated second term of Equation \ref{eq:prop_diff_bound}.
The third sum is bounded similarly by $\tfrac{2 M}{p-1} (N^\beta -1)^{1-p}  e^{-\mathcal{C}(\beta,1)t}$, the proposition follows from adding the
    bounds for all three sums.
\end{proof}

We now state the main result of the paper, which shows that in general
$z_k(t)$ is well approximated by two travelling waves, with two
different signal velocities, in opposite directions.

\begin{theorem} \label{theo:timedomain1} Let $S_N^*$ be as in Theorem
  \ref{theo:main1}, $c_\pm\equiv c_{0\pm}$ as given in Lemma
  \ref{lem:phasevelocity2}. Fix $0<\alpha < \beta < 1$. Let
  $|a_m|<Mm^{-p}$ and $|b_m|<Mm^{-p}$ for $p>1$, where $a_m$ and $b_m$ are as
  in Equation \ref{eq:zkt_solution_expansion}. Fix $K>1$. Define
  \begin{equation*}
    \mathcal{C}(a,b) = \min_{\substack{N^a \leq |m| \leq \min(N/2,N^b)\\ \varepsilon\in\{-1,1\}}}
	|\Re(\nu_{m\varepsilon})|
  \end{equation*}
  Then, for sufficiently large $N$, there are functions $f_-$ and
  $f_+$, and constant $D$ so that
  \begin{align} \label{eq:maintheorem}
    |z_k(t) - f_-(k-c_-t) - f_+(k-c_+ t) | & \leq 
    {MDK}(\tfrac{1}{|c_-|}+\tfrac{1}{c_+}) N^{3\alpha-1}  \notag \\
    &+ \tfrac{4M}{p-1}( (N^\beta-1)^{1-p}-(N^\alpha-1)^{1-p})e^{-\mathcal{C}(\alpha,\beta)t}  \notag \\
& +\tfrac{4M}{p-1}( (N^\beta-1)^{1-p} )  e^{-\mathcal{C}(\beta,1) t}
  \end{align}
  for all $t\in[\tfrac{N}{|c_-|},K\tfrac{N}{|c_-|}] \cap [\tfrac{N}{c_+},K\tfrac{N}{c_+}]$.

  In addition, if $\alpha <1/3$, then all terms on the r.h.s. of the
  above inequality tend to 0 as $N\to\infty$.
\end{theorem}
\begin{proof}
  An analogous result to Proposition \ref{prop:timedomain1} can be
  proved for the case when $a_m=0$ for all $m$. If we write $z_k(t) =
  z^+_k(t)+z^-_k(t)$, where $z^-_k(t)$ has expansion with all $b_m=0$
  and $z^+_k(t)$ has expansion with all $a_m=0$, we have 
\begin{equation*}
|z_k(t) -
  f_-(k-c_-t)-f_+(k-c_+t)| \leq |z^+_k(t) - f_+(k-c_+t)| + |z^-_k(t) -
  f_+(k-c_-t)| .
\end{equation*}
Using Proposition \ref{prop:timedomain1}  and the aforementioned
analogous result to bound the two terms on the right establishes
Equation \ref{eq:maintheorem}.

If $\alpha < 1/3$, then $N^{3\alpha-1} \to 0$ as $N \to \infty$. For $p>1$,
both  $(N^\beta-1)^{1-p} \to 0$ and  $(N^\alpha-1)^{1-p} \to 0$, as $N\to\infty$, which proves that 
all terms on the r.h.s of Equation \ref{eq:maintheorem} go to zero as $N\to\infty$.




\end{proof}

\begin{figure}[pbth]
\center
\includegraphics[height=3.0in]{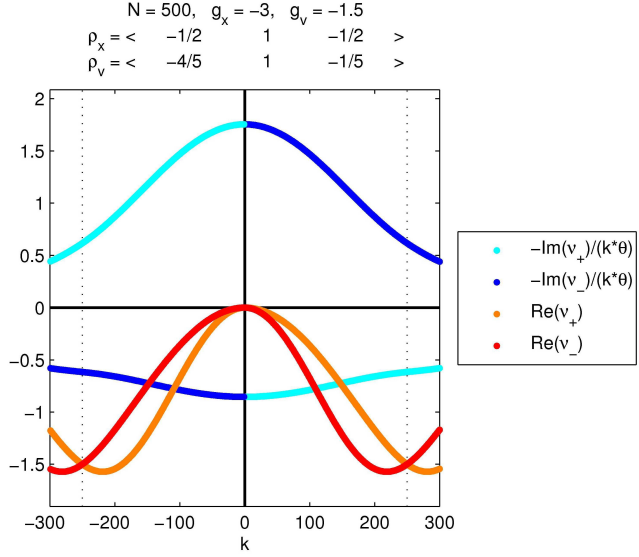}
\caption{(color online) \emph{A representative figure for the calculation of the phase velocities
for 500 agents.
The values of the parameters are given in the figures. Light blue: $\dfrac{-\Im(\nu_{m+})}{m\theta}$, Blue: $\dfrac{-\Im(\nu_{m-})}{m\theta}$,
Orange:  $\Re(\nu_{m+})$, Red: $\Re(\nu_{m-})$. The maximum phase velocities occur at $m=0$,
These are the signal velocities $c_+>0$ and $c_-<0$ of Theorem \ref{theo:timedomain1}. }}
\label{fig:phasevelocities3}
\end{figure}

{\noindent \bf Remark:} It is interesting to note that the signal velocity we determine
is actually equal to the group velocity at $m=0$. The group velocity
is defined as $\dfrac{dc_{m\pm}}{d(m\theta)}$. It is not necessarily
true that group velocity in these kinds of systems equals signal velocity. In the system studied
in \cite{WW} they are different. See \cite{Brillouin} for more information.

{\noindent \bf Remark:} A similar argument as the one in Theorem \ref{theo:timedomain1} easily shows that
eigenfunctions with wave numbers $m$ greater than $N^{0.5+\delta}$ will die out
before $t=N/c_+$. Thus for considerations on time-scales longer than that, these are
irrelevant. It also (conveniently) turns out that very often the greatest
phase velocities are associated with the lowest wave numbers. A typical case
is seen in Figure \ref{fig:phasevelocities3}. One can show that
in those asymptotically stable cases where $\rho_{v,1}$ is close to -1/2, we have
that $c_{m\pm}$ has a local maximum at $m=0$. In fact Lemma \ref{lem:phasevelocity2}
implies that for $\rho_{v,1}=-1/2$:
\begin{equation*}
c_{m\varepsilon}=\varepsilon\,\sqrt{\dfrac{-g_x}{2}}+
\varepsilon m^2\theta^2 \left(\dfrac{g_x-\frac 32 g_v^2}{24\sqrt{-2g_x}} \right)+\cdots
\end{equation*}
which has a local maximum at $m=0$.

\section{Conclusion} \label{chap:four}

Though experiments with cars have been done on circular roads (see \cite{S+}), our interest in
the system with periodic boundary conditions of $S_N^*$ as defined in Definition
\ref{defn:normalized system} stems from the applicability
to traffic systems with non-periodic boundary conditions.
The primary motivation for studying the former is that they enable us to analyze how disturbances
propagate, and --- under the assumption that this propagation does not depend
on boundary conditions --- apply that to the latter systems to find the transients.
Some remarks on how that works are given in the Introduction and is the subject of \cite{CV2}.
A relative novelty here is that we consider all strictly decentralized systems,
not just symmetric ones.

In Section \ref{chap:stability} we give precise conditions on the parameters
so that decentralized systems with periodic boundary condition are asymptotically stable.
In its generality stated here this is new, though related observations have been made in
\cite{BMH} and \cite{LFJ}. The main importance here is that we use these conditions
on the parameters to show that in these systems disturbances travel with constant a
constant signal velocity, and --- as our main result --- we determine that velocity
in Section \ref{chap:signal}. This explains why in these cases, approximations of these systems
with large $N$, by the wave equation are successful (see for example \cite{BMH}).
It can be shown however that for other parameter values diffusive behavior may occur (see \cite{CV2}.

\begin{figure}[t]
\center
\includegraphics[bb=0 0 447 401,height=3.0in]{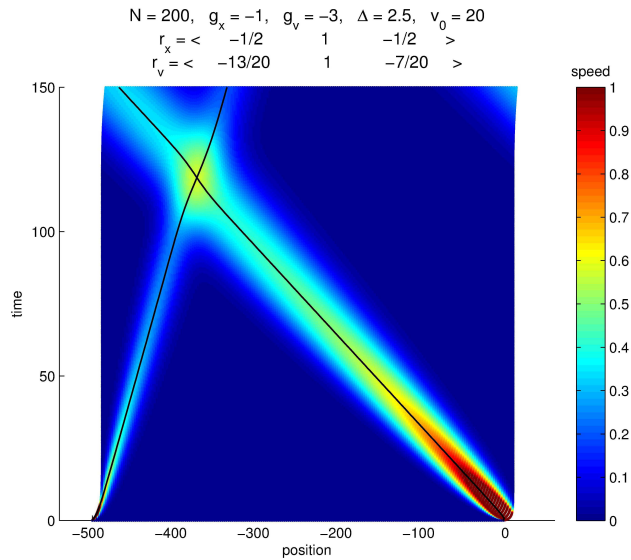}   
\caption{(color online) \emph{The orbits of 200 cars with specific choices for the parameters.
($\Delta$ is the desired distance between cars.)  At time 0 agent 0 receives a different initial condition.
They are color coded according to the velocity of the agent.  The black curves indicate the theoretical
position of the wavefront calculated via the signal velocity. Note that these velocities depend on the
direction, and that the signal velocity is measured in number of cars per time unit. Due to the different
velocities of the cars, these curves are not straight lines.}}
\label{fig:signalvelocity}
\end{figure}

Finally we test our prediction of the signal velocity in a numerical experiment.
Our theory described the error due to approximating the disturbance signal
as having a pair of signal velocities $c_{\pm}$ as a sum of  three terms (see equation \ref{eq:maintheorem}), which asymptotically go to zero for large $N$,
subject to a constraint on the decay of the Fourier coefficients of the initial disturbance.
In this numerical experiment we give agent number $0=N$ at time $t=0$ is a different initial velocity from the others. We note that even though
this type of impulse disturbance does not have the Fourier coefficient decay required by our theory, we nonetheless observe two distinct signal velocities
as predicted.
The result can be seen in Figure \ref{fig:signalvelocity}. That signal
propagates forward (in the direction 1,2,3,..) through the flock as well
as backwards (in the direction $N-1$, $N-2$, $N-3$,...). In figure we color coded
according to the speed of the agents, who are stationary until the signal reaches
them. In black we mark when the signal is predicted to arrive, according to the theoretically predicted signal velocities. One can see the
excellent agreement.


\end{document}